\documentclass[11pt,a4paper]{article}
\usepackage[english]{babel}
\usepackage{amsmath,amssymb,amsthm,epsfig,color,graphicx}
\usepackage[latin1]{inputenc}
\usepackage{tikz}
  \newtheorem*{theorem*}      {Theorem}
	\newtheorem*{conjecture*}   {Conjecture}
  \newtheorem{theorem}        {Theorem}
  \newtheorem{lemma}          {Lemma}
  \newtheorem*{lemma*}        {Lemma}
  \newtheorem{definition}     {Definition}
  \newtheorem{corollary}      {Corollary}
  \newtheorem{proposition}    {Proposition}
  \newtheorem{notation}       {Notation}

\definecolor{Red}{cmyk}{0,1,1,0}

\definecolor{Blue}{cmyk}{1,1,0,0}

\newcommand{\ba}{\begin{array}}
\newcommand{\ea}{\end{array}}
\newcommand{\be}{\begin{equation}}
\newcommand{\ee}{\end{equation}}
\newcommand{\ben}{\begin{enumerate}}
\newcommand{\een}{\end{enumerate}}

  \def\d{\mathop{\textrm{\rm d}}\nolimits}                  %d
  \def\exp{\mathop{\textrm{\rm exp}}\nolimits}              %exp
\newcommand{\calb}{\ensuremath        {\mathcal{B}}}
\newcommand{\calf}{\ensuremath        {\mathcal{F}}}

\let\a=\alpha

\let\g=\gamma

\tikzset{
  treenode/.style = {align=center, inner sep=0pt, text centered,
    font=\sffamily},
  arn_r/.style = {treenode, circle, blue, draw=blue, 
    text width=1em, very thick},% arbre rouge noir, noeud rouge
}
\begin{document}

\title{Counting Contours on Trees }
\author{Noga Alon\thanks{Research supported in part by BSF grant
2012/107 and by ISF grant 620/13.}
\\
\footnotesize{Sackler School of Mathematics and
Blavatnik School of Computer Science, Tel Aviv University}\\
\footnotesize{\texttt{nogaa@post.tau.ac.il}}
\\[0.3cm]
Rodrigo Bissacot\thanks{Partially supported by the Dutch stochastics 
cluster STAR (Stochastics - Theoretical and Applied Research), 
also supported by FAPESP Grants 11/16265-8, 2016/08518-7 and 
CNPq Grants 486819/2013-2, 312112/2015-7.}\\
\footnotesize{Institute of Mathematics and Statistics - IME-USP - 
University of S\~ao Paulo}\\
\footnotesize{\texttt{rodrigo.bissacot@gmail.com}}
\\[0.3cm]
Eric Ossami Endo\thanks{Supported by FAPESP Grants 14/10637-9 
and 15/14434-8.} \\
\footnotesize{Institute of Mathematics and Statistics - IME-USP - 
University of S\~ao Paulo}\\
\footnotesize{Johann Bernoulli Institute for Mathematics and Computer 
Science - University of Groningen}\\
\footnotesize{\texttt{eric@ime.usp.br}}\\
}
\maketitle

\begin{abstract}
We calculate the exact number of contours of size $n$ containing a fixed
vertex in $d$-ary trees and provide sharp estimates for this number
for more general trees. We also obtain a characterization of the locally
finite trees with infinitely many contours of the same size
containing a fixed vertex.
\end{abstract}



\section*{Introduction}
\label{intro}

After the seminal paper of Rudolf Peierls \cite{Pe}, the standard
technique to prove the existence of phase transitions in spin
systems (Ising model type, for instance) goes by a \textit{contour
argument}. Roughly speaking, we need to define objects usually called
\textit{contours}, notions of size (length or surface) and interior for
these objects. Furthermore, for a fixed vertex $x_0$ of a graph $G$ and,
for each $n \in \mathbb{N}$, we need to estimate the number of contours
of size $n$ in $G$ with $x_0$ in their interiors.

A standard calculation in this approach is to control expressions as below:

\begin{equation}\label{fundamental}
\sum_{ C \odot x_0} w(|C|) = \sum_{n=1}^{\infty} \sum_{\substack{ C \odot x_0\\ |C| = n}} w(|C|) = \sum_{n=1}^{\infty} w(n)  \sum_{ \substack{C \odot x_0 \\ |C| = n}} 1,
\end{equation}
where $|C|$ denotes the size of the contour $C$ and $C \odot x_0$
denotes the fact that $x_0$ belongs to the interior of $C$. Usually
the function $w: \{ \text{contours} \} \rightarrow \mathbb{R}^{+}$
depends only on the size of the contour and not on its position in the
graph $G$. For the standard Ising model on $\mathbb{Z}^2$, the function
is given by $w(C)= w(|C|) = \exp(- 2 \beta |C|)$ where $\beta$ is the
inverse of the temperature. Then, to control (\ref{fundamental}) we
need to estimate $\sum_{\substack{C \odot x_0 \\ |C| = n}} 1$ and for
this purpose generating functions are very powerful tools. We can find
similar expressions to (\ref{fundamental}) in almost all papers using
the \textit{Peierls argument}. The readers interested in the proof of
the existence of phase transition using the Peierls contours can check
standard books on the field \cite{Bo, Ge, FiSe, Ru}. The original Peierls
argument \cite{Pe} was done for the Ising model on $\mathbb{Z}^2$,
but we can define contours for any $\mathbb{Z}^d$ with $d \geq 3$ and
the argument works as well. The estimates of the number of contours
help us to give bounds for the critical temperature of the models, see
\cite{BB07, LM98}. These facts show that the mathematical problem of
counting contours on graphs has important consequences in statistical
physics and naturally emerges.

Moreover, the problem of counting finite objects on graphs (subgraphs,
paths with a fixed length, etc) is important for mathematicians and it
is a classical problem in discrete mathematics. The history about the
question of counting contours of the same size containing a fixed unit
cube on  $\mathbb{Z}^d$ ($d \geq 2$) is the following: David Ruelle proved
that there exist at most $3^n$ contours of size $n$ containing a fixed
unit cube; Lebowitz and Mazel \cite{LM98} proved that there are between
\((C_1d)^{n/2d}\) and \((C_2d)^{64n/d}\); and finally, differently from
the previous approaches and using \emph{generating functions}, Balister
and Bollob\'as \cite{BB07} improved these bounds showing that
there are between \((C_3d)^{n/d}\) and \((C_4d)^{2n/d}\) contours of
size \(n\) ($C_1, C_2, C_3 \ \text{and} \ C_4$ are constants).

In the last years, some attention was given to the Ising model on trees
instead of $\mathbb{Z}^d$, and there is more than one definition of
contour for trees and general graphs \cite{ APS12, BB99, LM98, Rozi,
Rozi1, Rozi2, Rozi3}.

In this note, we consider a definition proposed by Babson and Benjamini
\cite{BB99}. We will see that this definition on trees implies
that the number of contours of size $n$ coincides with the number of
\textit{external boundaries} with $n$ vertices, a standard notion
used by the combinatorics community. In the original paper, they used the
term \textit{cut sets} as is usual for combinatorialists, the context
was percolation theory, see also \cite{BeB99}. This definition 
was later 
considered in \cite{APS12} in the study of bounds for the critical
percolation probability $p_c$ in general graphs.

Our contribution is to clarify the connection between contours on trees
and natural objects in graph theory. Inspired by Balister and Bollob\'as
\cite{BB07}, we show that in the case of regular trees (and $d$-ary trees)
we can calculate the exact number of contours of size $n$ containing
a fixed vertex $x_0$. We also obtain a characterization for locally
finite rooted trees with infinitely many contours of some fixed size $n$
involving the root. In particular, we prove that we have infinitely many
contours of the same size if and only if the tree has an {\it infinite
independent path}. Nonamenable trees are the trees in which the length
of the independent paths is uniformly bounded. In particular, trees
which contain an infinite independent path are amenable trees. On the
other hand, for nonamenable graphs with bounded degree, (in particular,
$d$-ary trees) one possibility for the proof of the phase
transition in
Ising models and for the study of ground states is to count the number
of connected components of a fixed size containing a vertex,
instead of
counting the number of contours, see \cite{GJRSS, JS}.

This note is organized as follows: in Section \ref{Definitions and
Notations} we give some basic definitions of graph theory, introduce
the precise definition of a contour, and show the connection of these
objects with external boundaries in graphs. In Section \ref{Contours
on d-ary and regular trees} we give explicit expressions for the number
of contours of size $n$ in regular and $d$-ary trees. In addition, we
show that the binary trees are extremal objects with respect to the number
of contours of a fixed size.  More precisely, if we fix $n$, the number
of contours of size $n$ containing a fixed vertex  
is maximum in binary trees when we consider
locally finite 
trees in which each vertex has at least two children. In Section
\ref{Infinitely many contours with size $n$} we give a geometric
characterization of trees with infinitely many contours of the same
size containing a fixed vertex. It turnes out  that this is equivalent to
the existence of what's called an infinite independent path
in the tree. 

\section{Definitions and Notations}\label{Definitions and Notations}

The graphs $G=(V,E)$ considered are always simple, undirected, connected,
with countably infinite number of vertices. All the graphs are locally
finite, in other words, with finite degree for each vertex of $V$. The
degree of a vertex $x$ is the number of edges which are incident to $x$,
denoted by $\d(x)$. A \textit{path} $\g$ is an alternating sequence of
vertices and edges $\gamma = (v_0, e_1, v_1, e_2, ..., e_k, v_k)$ where
$e_i = v_{i-1}v_i=: \{v_{i-1}, v_i \}$ and all vertices are
distinct, with the possible exception of $v_0,v_k$.
The vertices $v_1, v_2, ..., v_{k-1}$ are called \textit{inner
vertices} of $\g$. An \emph{independent path}~$\gamma$ in a graph ~$G$ is
a path where all inner vertices of~$\gamma$ have degree two in~$G$. When
$v_0 = v_k$ we say that the path $\g$ is a \textit{cycle}. We say that
a graph $G$ is a \textit{tree} if it is connected and has no cycles.

Given a vertex $x$ and a 
subset of vertices $A \subset V$, let $\d_G(x,A)$ denote the number
$
\d_G(x,A) = \min \{ |\gamma| ; \gamma \ \text{is a path in}\ G \ 
\text{connecting}\  x \ \text{to a vertex of} \ A \},
$
where for each path $\gamma$ in $G$, $|\gamma |$ denotes the number of
edges of $\gamma$. Thus $\d_G(x,A)$ is the usual distance in the graph $G$
between $x$ and $A$. The set $\partial_v^{ext}A=\{x \in V\setminus A:\
\d_G(x, A)=1\}$ is the \emph{external boundary} of $A$.

Let $G=(V,E)$ be a graph, we say that a graph $\tilde{G}$ is a
\textit{minor} of $G$, denoting by $\tilde{G} \preceq G$, when
$\tilde{G}$ is obtained from $G$ after a sequence of the following
operations: contracting some edges, deleting some edges and/or isolated
vertices. We \textit{contract} an edge $e=xy$ and obtain a graph that we
denote by  $G/e$ when we delete the edge  $e$ from $E$, add to $E$
the
collection of edges $\{ az; xz \in E \ \text{or} \ yz \in E\}$ where $a$
is a new vertex replacing the vertices $x$ and $y$, and
remove all resulting parallel edges.
Thus $V(G/e) = V(G)\backslash(\{x,y\}) \cup \{a\}$.
We $\textit{delete}$ an edge $e=xy$ when we
remove the edge from the graph but keep the vertices on it, after the
process we obtain a new graph $G\setminus e = (V, E \setminus \{e\})$,
for a finite collection of edges $C$ the procedure is the same,
keeping the
vertices and deleting the edges: $G\setminus C = (V,E \setminus C ) $.

\begin{definition}
Given a graph~$G=(V,E)$, a finite set~$C \subset E$ is called
a \emph{contour} if~$G\setminus C$ has exactly one finite connected
component, and it is minimal with respect to this property. That is, for
all edges~$e\in C$ the graph~$(V,E\setminus(C\setminus e))$ does not
have a finite connected component.

If~$C$ is a contour in~$G$ then we denote by~$G_{C}=(I_{C},E_{C})$
the unique finite connected component of~$G\setminus C$.
\end{definition}

This notion was originally defined by Babson and Benjamini in
\cite{BB99} where the authors used \textit{minimal cut set} instead
of \textit{contour}. The definition was later used in \cite{APS12} in
the study of
percolation problems on graphs.

Let~$\calf_G$ be the set of all contours of~$G$.  We denote
by~$\calf_G^n$ the set of all contours of $G$ of size $n$; by~$\calf_G(x)$
($\calf_G^n(x)$) the set of all contours~$C\in \calf_G$ ($C \in
\calf_G^n$) such that $x \in I_{C}$.

Let~$T_d$ be a rooted tree such that all vertices have~$d$ children, i.e.,
the root has degree~$d$ and the other vertices have degree~$d+1$. The
tree~$T_d$ is called  
\emph{d-ary tree}. A~$2$-tree is called \emph{binary
tree}. When all the vertices of a tree have the same degree $d$ we say
that the tree is a $d$\textit{-regular tree}.

\begin{center}
\begin{tikzpicture}[line width=1pt,level/.style={sibling distance 
= 3cm/#1, level distance = 1cm}] 
\node {\(x\)}
    child{
            child{
            	child{edge from parent [green]}
							child{edge from parent [green]}
            }
            child{ edge from parent [green]
							child { edge from parent [black]}
							child{edge from parent [black]}
            }                            
    }
    child{edge from parent [green]
            child{edge from parent [black]
							child{edge from parent [black]}
							child{edge from parent [black]}
            }
            child{edge from parent [black]
							child{edge from parent [black]}
							child{edge from parent [black]}
            }
		}
; 
\end{tikzpicture}
\end{center}
\begin{center}
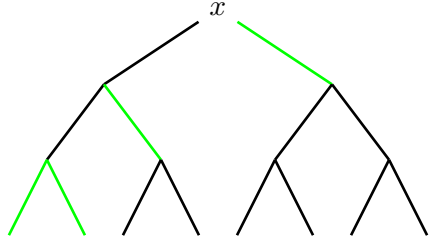
\textit{Example of a  contour of size four in a binary tree $T_2$}
\end{center}

We finish this section showing that in the case of trees there is a
one-to-one relation between contours of size $n$ and 
%vertex
external boundaries
of size $n$. This proposition will allow us to conclude that for
binary trees the number of contours of size $n$ containing the root is
the $n$-th Catalan number.

\begin{proposition}\label{catalan}

 Let $T=(V,E)$ be a rooted, locally finite and infinite tree. Let
 $x_0$ the root and suppose that $T$ does not have leaves. Let 
$$ 
\calb^n_{T}(x_0)=\{ B\subset V: B \text{ is finite, connected, }x_0\in
 B \text{ and }|\partial^{ext}_v(B)|=n \} 
$$
be the set of finite subtrees
(induced by the vertices) of $T$ containing $x_0$ with external
boundary of size $n$.  Then there is a bijection between
 $\calb^n_{T}(x_0)$ and $\calf^n_{T}(x_0)$.
\end{proposition}

\begin{proof}
We will prove that for each $B \in \calb^n_{T}(x_0)$ there exists a
contour $C$ such that $\partial^{ext}_v(B) = C$.  We define the function
$f:\calf^n_{T}(x_0)\to \calb^n_{T}(x_0)$ in the following way: let $C$
be a contour in $\calf^n_{T}(x_0)$. Remove all edges of $C$ from the tree
$T$. By definition of contour, we get a finite connected component $B$
containing $x_0$.  Define $f(C)=B$. To show that $f$ is well defined,
we shall prove that $|\partial^{ext}_v(B)|=n$. Actually, $B=I_{C}$.

By definition of contour, each edge in $C$ has one endpoint in $B$ and
the other in $V\setminus B$. Let $V_{e}(C)$ be the set of endpoints in
$C \cap (V\setminus B)$. As $|C|=n$ and the graph is a tree, we have
$|V_{e}(C)|=n$.  Clearly $V_{e}(C)\subseteq \partial^{ext}_v(B)$. If
some element $u$ of $\partial^{ext}_v(B)$ does not belong to $V_{e}(C)$,
the edge connecting $u$ with $B$ does not belong to $C$, contradicting
the fact that $C$ is a contour. Thus $f$ is well defined.  It is not hard
to see that $f$ is a bijective function.
\end{proof}

\section{Contours on d-ary and 
regular trees}\label{Contours on d-ary and regular trees}

The main technique in 
this note is the use of generating functions in the investigation
of  counting
problems on trees; this approach produces very clean proofs. Classical
references to the technique are the two books of Richard P. Stanley
\cite{S1, S}.

The well known \textit{Catalan numbers}
$C_{n-1} =\frac{1}{n}\binom{2n-2}{n-1} (n \in \mathbb{N})$, have
lots of interpretations in Combinatorics. In particular, see, e.g.,
\cite{S1,S}, these numbers count the number of contours
in binary trees by Proposition \ref{catalan}. In fact,  let ~$T_2$
be the binary tree with root~$x_0$. For all $n\ge 2$, we have
$|\calf_{T_2}^n(x_0)|=\frac{1}{n}\binom{2n-2}{n-1}$.

Here we present a proof where we calculate the exact number of contours
in $d$-ary trees using generating functions, an alternative 
derivation can be found
in \cite{S}. Let~$\mathbb{R}(\!(z)\!)$ be the ring of formal series
defined by
$$
\mathbb{R}(\!(z)\!)=\left\{ \sum_{k\ge 0}a_kz^k:  
a_k\in \mathbb{R} \right\}.
$$

We define the operator~$[z^n]$ 
which extracts the coefficient of $z^n$ in the series, that is, 
$[z^n](\sum_{k\ge 0}a_kz^k)=a_n$.

The Lagrange Inversion Theorem states that we can compute exactly the
coefficients of a series under certain conditions. The reader can find
a proof of this theorem in \cite{FS09, S}.

\begin{theorem*}[Lagrange Inversion Theorem, Lagrange -- 1770]\label{lagrange}\quad
Let $\phi \in \mathbb{R}(\!(z)\!)$ with $\phi(0)\neq 0$  and ~$f(z)\in z\mathbb{R}(\!(z)\!)$ defined by ~$f(z)=z\phi(f(z))$, then
$$
[z^n]f(z)=[z^{n-1}]\frac{1}{n}\phi(z)^n.
$$
\end{theorem*}

\begin{proposition}\label{darytree}
Let~$d\ge 2$,~$n\ge 1$,~$T_d$ be a~$d$-ary tree with root~$x_0$. Then
$|\calf_{T_d}^{1}(x_0)|=0$ and
$$
|\calf_{T_d}^{n}(x_0)|=
\begin{cases}
\frac{1}{n}\binom{\frac{d}{d-1}(n-1)}{\frac{1}{d-1}(n-1)},& \text{ if }n\equiv 1 \pmod{d-1};\\
0, & \text{ otherwise},
\end{cases}
$$
when~$n\ge 2$.
\end{proposition}

\begin{proof}

For each edge with endvertex~$x_0$, we can either include this edge
in the contour or not. If we do not include it, we carry the root~$x_0$
to the other endvertex of this edge and apply again the same
procedure. Consider~$f(X)=\sum_{n\ge 1}a_nX^n$ the generating function
in which the coefficients are~$a_n=|\calf_{T_d}^n(x_0)|$ for all $n\ge
1$. Then we have the following equation $f(X)=(X+f(X))^d$.
Consider~$h(X)=X+f(X)$, we have~$h(X)=X+h(X)^d$, so 
$h(X)=X(1-h(X)^{d-1})^{-1}$.  Applying Lagrange's Theorem  
with~$\phi(X)=(1-X^{d-1})^{-1}$ we obtain 
$[X^n]h(X)=\frac{1}{n}[X^{n-1}]\phi(X)^n.$
Now,
$$
\phi(X)^n=(1-X^{d-1})^{-n}=\sum_{k\ge 0}\binom{n+k-1}{k}X^{(d-1)k}.
$$
Thus, if~$n-1=(d-1)k$ for some~$k$, we have
$$
[X^n]h(X)=\frac{1}{n}\binom{n+k-1}{k}=\frac{1}{n}\binom{\frac{d}{d-1}(n-1)}{\frac{1}{d-1}(n-1)}.
$$
\end{proof}

\textit{Remark:}
There is a geometric interpretation for the
equation~$h(X)=X+h(X)^d$. Let~$T_d$ be a~$d$-ary tree with root~$x_0$. Add
an edge~$e$ for which~$x_0$ will be a leaf, and it will be an endpoint
of $e$. Now we can either include the edge~$e$ in the contour or not. If we
do not include it, 
we carry the root~$x_0$ to the other endvertex of this edge
and apply the same procedure again.

\begin{corollary}
Let~$d\ge 2$, $T_d$ be a~$d$-ary tree with root~$x_0$, 
and let~$n\ge 1$ be such that $n\equiv 1 \pmod{d-1}$, and $k=(n-1)/(d-1)$. We have
$$
\frac{1}{n}d^k \le |\calf_{T_d}^{n}(x_0)|\le \frac{1}{n}(ed)^k.
$$
\end{corollary} 

\begin{proof}
This is consequence of 
the following inequality. For all integers $0\le k\le n$,
$$
\left(\frac{n}{k}\right)^k \le \binom{n}{k}\le \left(\frac{en}{k}\right)^k.
$$
\end{proof}

\begin{proposition}\label{regular_tree}
Let~$d\ge 2$,~$n\ge 1$,~$T$ be a~$(d+1)$-regular tree, and $x$ be a vertex of $T$. Then
$$
|\calf_T^n(x)|=a_{n-1}+\sum_{k=1}^{n-1}a_ka_{n-k},
$$
where~$a_n=|\calf_{T_d}^n(x)|$.
\end{proposition}

\begin{proof}
Let~$g(X)=\sum_{n\ge 1}b_nX^n$ be the generating function with
coefficients~$b_n=|\calf_{T}^n(x)|$, and~$f(X)=\sum_{n\ge 1}a_nX^n$
be the generating function with coefficients~$a_n=|\calf_{T_d}^n(x)|$.
Note that~$g(X)=(X+f(X))^{d+1}=Xf(X)+f(X)^2$. The proof is a direct
consequence of the previous proposition.
\end{proof}

A natural question is to compare the number of contours between
different infinite trees. We next show that the
binary tree is extremal in the class of all locally finite trees
in which every vertex has at least two children.

\begin{theorem}\label{twochildren}
Let $T$ be a locally finite and infinite rooted tree. Let $x$ be the
root of $T$ and suppose that all vertices in \(T\) have at least
two children. Then, for all $n\ge 1$, we have $|\calf_T^n(x)|\le
|\calf_{T_2}^n(x)|.$
\end{theorem}

\begin{proof}

We will construct a binary labeled tree $T'$ such that $T$ is a minor
of $T'$ as follows. Starting from $x$ we process the vertices of $T$
according to a Breadth-First-Search order, that is, we start from the
root $x$, then process its neighbors, followed by their neighbors
and so on. When we process a vertex $y$ of $T$ that has $s > 2$
children, say $z_1,z_2,\ldots,z_s$, we replace $y$ by $s-1$ vertices
$y_1,y_2,\ldots,y_{s-1}$. For each $i$, the children of $y_i$ are
$y_{i+1}$ and $z_{i}$, and the children of $y_{s-1}$ are $z_{s-1}$
and $z_s$. When a vertex $y$ of $T$ has $2$ children, we keep the
vertex $y$. Clearly $T'$ is a binary tree. We call $x'$ the root of
$T'$. We will show that there exists an injective map $f$ which takes
each contour $C$ in $\mathcal{F}^n_T(x)$ and produces a contour $f(C)$
in $\mathcal{F}^n_{T'}(x')$. In fact, for each edge of the form $yz_i$ in
$C$, we associate the edge $y_iz_i$ in $T'$ (for $yz_s$ take $y_{s-1}z_s$)
and for $y$ with $s=2$ children we keep the edge $yz_i$. The collection of
edges produced by this procedure is defined as $f(C)$. To simplify the
argument let us call the new edges green edges, see the picture below.

We should prove that $f: \calf_T^n(x) \rightarrow
\calf_{T'}^n(x')$, in other words, that $f(C)$ belongs to
$\calf_{T'}^n(x)$. To see that $f(C)$ is a contour observe that, by
construction, there are no green edges in $f(C)$. Suppose by contradiction
that $T'\backslash f(C)$ has no finite connected component containing
the root $x'$, then there exists an infinite path $\g'$ in $T'$ starting
at the root $x'$ of $T'$. When we contract all the green edges in  $T'$,
in particular in $\g'$, we obtain the original tree $T$ and a path $\g$
in $T$ starting in the root $x$. Since there are no green edges in the path
$\g'$, we have now an infinite path $\g$ in $T$ starting at the root $x$
with $E(\g) \cap C = \emptyset$, a contradiction. To see that $f(C)$ has
the minimality property suppose that there exists an edge $e' \in f(C)$
such that $E(T') \backslash (f(C) \backslash \{e'\})$ still has a finite
component containing the root $x'$. When we contract all the green edges
and add the corresponding edge $e \in C$ (the edge associated to $e'$ by
$f$), since $C$ is a contour, there exists an infinite path $\a$ starting
at the root $x$ in $T$ such that $e \in E(\alpha)$. We will construct,
using the path $\alpha$, an infinite path $\alpha'$ in $T'$ starting
at $x'$ such that $e' \in E(\alpha')$, to get a contradiction. 
Indeed, consider the
process to construct the tree $T'$ in the vertices of $\alpha$. Starting
at the root $x$, for each edge $zy \in E(\alpha)$, where $z$ is the 
father of $y$,
after processing $z$ there exists $1\leq j \leq s-1$ such that $z_j y$
is an edge of $T'$. Add $z_j y$ to $E(\alpha')$. If $j=1$ we add the
edge $z_1 y$ to $\alpha$, if $j> 1$ we add the finite path starting
in $z_1$ and ending in $z_j$,  (which consists of green 
edges: $z_1z_2, z_2z_3,
..., z_{j-1}z_j$) and the edge $z_jy$ to $\alpha'$. Since the path $\a$
is infinite and $e \in E(\alpha)$ we construct an infinite path $\a'$,
starting in $x'$ such that $e'$ belongs to $\a'$. This shows that
$f(C)$ is indeed a contour.

It is also easy to check
that $f$ is injective and that $|C|=n$ implies $|f(C)|=n$.

\begin{center}
\begin{tikzpicture}[line width=1pt,
level 1/.style={sibling distance=1.5cm},
level 2/.style={sibling distance=0.5cm}
scale=0.50] 
\node {\(x\)}
    child{ node {\(y\)}
    }
    child{ node {\(z\)}
    }
    child{ node {\(w\)}
    }
;
\end{tikzpicture}
\begin{tikzpicture}[node distance=5cm,auto, scale=0.50]
\node (s) at (0,0) {};
\node (a) at (0,2) {};
\node (b) at (2,2) {};
\draw[thick,->, bend left=45](a) edge (b);
\end{tikzpicture}
\begin{tikzpicture}[line width=1pt,
level 1/.style={sibling distance=4cm},
level 2/.style={sibling distance=2cm},
level 3/.style={sibling distance=1cm},
scale=0.8pt] 
\node (Root) {\(x_1\)}
    child{node {\(y\)}             
    }
    child{node {\(x_2\)} edge from parent [green]
		            child{node[black] {\(z\)} edge from parent [black]
            }
            child{node[black] {\(w\)}  edge from parent [black]
            }
		}
; 
\end{tikzpicture}
\end{center}
\begin{center}
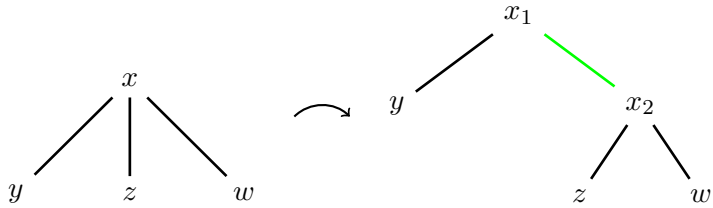
\textit{Example first iteration}
\end{center}

\begin{center}
\begin{tikzpicture}[line width=1pt,
level 1/.style={sibling distance=1.5cm},
level 2/.style={sibling distance=0.5cm},
scale=0.8pt] 
\node  {\(x\)}
    child{ node {\(y\)}
            child{ node {\(u_1\)}
            }
            child{ node {\(u_2\)}
            }
						child{ node {\(u_3\)}
            }  
    }
    child{ node {\(z\)}
            child{
            }
            child{ edge from parent [black]
            }
						child{ edge from parent [black]
            }  
    }
    child{ node {\(w\)}
            child{
            }
            child{ edge from parent [black]
            }
						child{ edge from parent [black]
            }  
    }
;
\end{tikzpicture}
\begin{tikzpicture}[node distance=5cm,auto,scale=0.50]
\node (s) at (0,0) {};
\node (a) at (0,3) {};
\node (b) at (2,3) {};
\draw[thick,->, bend left=45](a) edge (b);
\end{tikzpicture}
\begin{tikzpicture}[line width=1pt,
level 1/.style={sibling distance=4cm},
level 2/.style={sibling distance=2cm},
level 3/.style={sibling distance=1cm},
scale=0.8pt] 
\node (Root) {\(x_1\)}
    child{node {\(y_1\)}
            child{ node {\(u_1\)}
            }
            child{ node {\(y_2\)} edge from parent [green]
							child { node[black] {\(u_2\)} edge from parent [black]}
							child{node[black] {\(u_3\)} edge from parent [black]}
            }                            
    }
    child{node {\(x_2\)} edge from parent [green]
            child{node[black] {\(z\)} edge from parent [black]
            }
            child{node[black] {\(w\)}  edge from parent [black]
            }
		}
; 
\end{tikzpicture}
\end{center}
\begin{center}
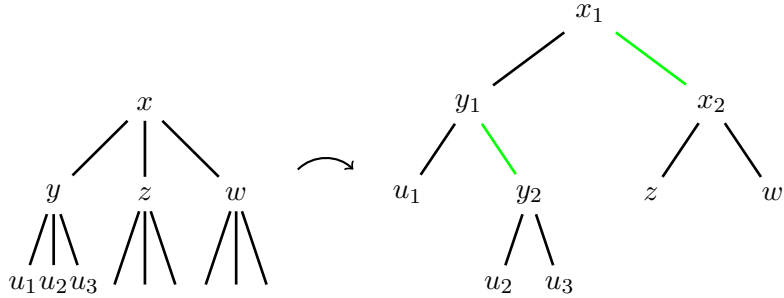
\textit{Part of the second iteration }
\end{center}

\end{proof}

By the theorem above we have an estimate for trees in which each vertex has
at least~$r$ children, where $r\ge 2$. However, we have a better estimate
for these trees. To prove this we use the inequality below, a classical
theorem in extremal combinatorics proved in \cite{Bollo}, see also
\cite{Alon85} and its references for several variants and
extensions.

\begin{theorem*}[Bollob\'{a}s, 1965]\label{bollobas}
Let $\{(A_i,B_i):i\in I\}$ be a finite collection of pairs of finite sets 
such that $A_i\cap B_j=\emptyset$ if and only if $i=j$. Then
$$
\sum_{i\in I}\binom{|A_i|+|B_i|}{|A_i|}^{-1}\le 1.
$$
In particular, if for all $i\in I$ we have $|A_i|\le a$ 
and $|B_i|\le b$, then
$$
|I|\le \binom{a+b}{a}.
$$
\end{theorem*}

\begin{theorem}\label{alon}
Let~$T$ be a locally finite infinite rooted tree with root~$x$. Suppose 
that all vertices of $T$ have at least $r$ children, $r\ge 2$. Then, 
for all $n\ge 1$,
$$
|\calf_T^n(x)|\le \binom{n+ \lfloor \frac{n-r}{r-1}\rfloor}{\lfloor  \frac{n-r}{r-1}\rfloor}.
$$

where $\lfloor x \rfloor =\max \{n \in \mathbb{Z} : n \leq x \}.$
\end{theorem}

\begin{proof}
Let $C$ be a contour of size $n$ in $T$ and let $I_C$ be the finite
connected component when we remove $C$ from $T$. We will find an upper
bound for the number of edges $|E(I_C)|$ in $I_C$. Let $B$ be the induced
subgraph of $T$ on the union of $I_C$ and $C$. Note that $B=(V,E)$ is
a rooted finite subtree of $T$ with $n$ leaves, and each vertex of $B$
that is not a leaf has at least $r$ children. Let $t$ be the number of
vertices of $B$ and consider the number $k=t-n$. Note  that $k$ is the
number of vertices in $I_C$. Using the fact that all vertices of $T$
have at least $r$ children, we have
$$
2(t-1)=\sum_{v\in V}\d(v)\ge (k-1)(r+1)+r+n.
$$
Thus, $k\le (n-1)/(r-1)$.

Since $I_C$ is a tree, the number of edges in $I_C$ is $|E(I_C)|=k-1\leq (n-r)/(r-1)$.

To finish the proof we need the following:

\textit{Fact:} If $C_1$ and $C_2$ are two distinct contours of a
vertex $x$ in $T$,
each
of size $n$, and if $I_{C_1}$ is the finite connected component when we
remove $C_1$ from $T$, then $E(I_{C_1}) \cap C_2\neq \emptyset$.

\textit{Proof of fact:} Suppose, by contradiction, that there exist
two contours $C_1$ and $C_2$ as above in $T$, each of
size $n$, such that $E(I_{C_1})
\cap C_2= \emptyset$. Let $I_{C_2}$ be the finite connected component
when we remove $C_2$ from $T$. Then $I_{C_1}$is a subgraph of $I_{C_2}$ and
$I_{C_1}\neq I_{C_2}$. Since $I_{C_1}$ and $I_{C_2}$ are subtrees of $T$,
and as all vertices in $T$ have at least $r$ children, we have
$|\partial_e(I_{C_1})|<|\partial_e(I_{C_2})|=n$, a contradiction. This
proves the fact.

Finally, let us prove the desired result. Let $(C,E(I_C))$ be a pair of
a contour $C$ of size $n$, where $I_C$ is the finite connected component
when we remove $C$ from $T$. We have $|C|=n$ and $|E(I_C)|\le \lfloor
(n-r)/(r-1)\rfloor$. The set $\{(C,E(I_C)):C\in \mathcal{F}^n_T(x)\}$
is finite, and $C_1\cap E(I_{C_2})=\emptyset$ if and only if $C_1=C_2$.
By the theorem above,
$$
|\calf_T^n(x)|\le \binom{n+ \lfloor 
\frac{n-r}{r-1}\rfloor}{\lfloor \frac{n-r}{r-1}\rfloor},
$$
concluding the result.
\end{proof}

It is sometimes desirable
to consider contours whose edges cover a fixed
vertex, see \cite{Rozi1}. We obtain some bounds for this case as
well.

\begin{definition}
Let $T$ be an infinite tree with root $x_0$. A \emph{rooted
contour} is a contour $C$ such that there exists an edge
$l\in C$ incident with the
root $x_0$.
\end{definition}

We denote by $\calf_{r,T}^n(x_0)$ the set of rooted contours $C$
on $T$ of size $n$. We can calculate exactly $|\calf_{r,T}^n(x_0)|$
for $d$-ary trees and regular trees. Clearly $|\calf_{r,T}^n(x_0)|\le
|\calf_{T}^n(x_0)|$.

\begin{proposition}
Let $T_d$ be a $d$-ary tree with root $x_0$. Then, for all $n\ge d$:
$$
|\calf_{r,T_d}^n(x_0)|=
a_n - \sum_{m_1+\ldots+m_{d}=n}a_{m_1}\ldots a_{m_{d}};
$$
where $a_n=|\calf_{T_d}^n(x_0)|$. \emph{(Note that $|\calf_{r,T_d}^n(x_0)|=0$ if $n < d$).}
\end{proposition}

\begin{proof}
Let $f_{T_d}(X)=\sum_{n\ge 1}a_nX^n$ and  $f(X)=\sum_{n\ge 1}c_nX^n$ be
the generating functions with coefficients  $a_n=|\calf_{T_d}^n(x_0)|$ and
$c_n=|\calf_{r,T_d}^n(x_0)|$ respectively. For each edge incident
with $x_0$
we can add it or not to the contour $C$. Repeating the same process as
we did in Proposition \ref{darytree}, if we do not add an edge we carry
the root to the other endpoint of this edge.  By the same proposition we
know $f_{T_d}(X)=(X + f_{T_d}(X))^{d}$. Thus
$$
f(X)=(X + f_{T_d}(X))^{d} - (f_{T_d}(X))^{d}=f_{T_d}(X) - (f_{T_d}(X))^{d}.
$$
\end{proof}

\begin{proposition}
Let $T$ be a $(d+1)$-regular tree with root $x_0$. Then, for $n\ge d+1$:
$$
|\calf_{r,T}^n(x_0)|=
b_n - \sum_{m_1+\ldots+m_{d+1}=n}a_{m_1}\ldots a_{m_{d+1}};
$$
\end{proposition}

\begin{proof}
Using a similar argument to the one used in the previous proof,
let
$f(X)=\sum_{n\ge 1}d_nX^n$ be the generating function with coefficients
$d_n=|\calf_{r,T}^n(x_0)|$ and let 
$g(X)=\sum_{n\ge 1}b_nX^n$ be the generating
function from Proposition \ref{regular_tree}.  Then $f(X)= g(X) -
(f_{T_d}(X))^{d+1}.$
\end{proof}

\section{Infinitely many contours 
of size $n$}\label{Infinitely many contours with size $n$}

A natural question is to study when we have infinitely many contours for
some size $n$ whose finite connected component contains a fixed vertex
$x_0$. We can characterize the trees with this property.

\begin{notation}\label{curto}
Let \(G=(V,E)\) be a graph. For each finite independent path $\g$ of~\(G\)
linking two vertices $x$ and $y$, remove all the edges (and inner
vertices) of $\g$ and add the edge $xy$. Denote this new graph that
is a minor of $G$, possibly with fewer edges, by $\tilde{G}$.
\end{notation}

In the next lemma and proposition the notation $\tilde{G}$ is used for
this special case of minor.

\begin{lemma}\label{finitecontour}
Let ~$T$ be a tree with  root~$x$ without leaves. Suppose that each
independent path of $T$ has finite length. Then~$|\calf_T^n(x)|<+\infty$
if and only if~$|\calf_{\tilde{T}}^n(x)|<+\infty$.
\end{lemma}

\begin{proof}
For each contour~$C=\{e_1\ldots,e_n\}$ of~$\tilde{T}$, the
contour is associated to a  (unique) family of independent
paths~$\{\gamma_1,\ldots,\gamma_n\}$ of~$T$. Then,
$$
\sum_{C\in \calf^n_{\tilde{T}}(x)}\prod_{i=1}^n | 
\gamma_i  | =|\calf_{T}^n(x)|.
$$
As the sum and the product are finite, 
we obtain~$|\calf_{T}^n(x)|<+\infty$. The converse is analogous.
\end{proof}

Thus we obtain the following characterization:

\begin{theorem}\label{infinitecontour}
Let~$T$ be a locally finite rooted tree with a root $x$ and without
leaves. Then there exists~$n\ge 1$ such that~$|\calf_T^n(x)|=+\infty$
if and only if ~$T$ has an infinite independent path.
\end{theorem}

\begin{proof} 
If we assume that $|\calf_T^n(x)|=+\infty$, the above is a
consequence of Lemma \ref{finitecontour} combined with
Theorem \ref{twochildren}. For the converse, take an
infinite independent path $\g$. Let~$C$ be a contour of~$T$
that contains an edge~$e$ of~$\gamma$. For all edge~$e'$
of~$\gamma\setminus\{e\}$, we have that~$C'=(C\setminus \{e\})\cup
\{e'\}$ is a contour of~$T$ and~$|C|=|C'|$. Therefore, taking~$n=|C|$
we obtain~$|\calf_T^n(x)|=+\infty$.
\end{proof}

\begin{proposition}
Let $T$ be an infinite, locally finite rooted tree with root $x_0$
without leaves. Suppose that $T$ has an infinite independent
path. Then there exists a sequence $(n_i)_{i\ge 1}$ such that
$|\calf^{n_i}_T(x_0)|=+\infty$ if and only if there is an infinite number
of vertices in $T$ with degree at least three.
\end{proposition}

\begin{proof}
Suppose that there exist only a finite number of vertices in $T$
with degree at least three. Take $\tilde{T}$ constructed as in 
Notation \ref{curto}.  If an independent path is infinite, we replace this
independent path by a leaf. This new tree we denote by $T^{'}$. Since $T$
has an infinite independent path, $T^{'}$ has at least one leaf. Moreover,
$T^{'}$ is a finite tree because $T$ has only a finite number of vertices
with degree at least three. Let $B$ be a subtree of $T^{'}$ such that $x_0
\in B$ and $B$ does not contain any leaf. Let $C$ be the set of external
boundary edges of $B$.  For each $C$ constructed in this way we obtain a
family of contours of $T$ of the same size and any contour in $T$
comes from
some external boundary edges for some such $B$.  
As we have a finite number of
subtrees of $T^{'}$ that do not contain leaves, there exists $n_0\ge 1$
such that for all $n\ge n_0$ we have $|\calf^{n}_T(x_0)|=0$.

For the converse, suppose that there exists an infinite number of vertices
in $T$ with degree at least three. Let $E_k$ be the set of edges whose
distance from $x_0$ is $k$.  $E_k$ is a contour. Since the number of
vertices in $T$ with degree at least three is infinite, the number
of edges in each $E_k$ in tending to infinity when we increase $k$. Let
$(k_i)_{i}$ be an increasing sequence of natural numbers such that $n_i
=  |E_{k_i}|$ is also an increasing sequence. Let $\g$ be an infinite
independent path. Note that there exists  $i_0$ such that $E_{k_i}$
contains an edge  $e_i$ of the infinite independent path $\g$ for all $i
\geq i_0$. Then, since we can replace $e_i$ by any other edge of $\g$
and obtain a new contour of the same size $n_i$, we have infinitely many
contours of size $n_i$.
\end{proof}

\section*{Final Remark}

The Peierls strategy to look for contours involving a vertex fails if $w$ in (\ref{fundamental})  depends only on the size of the contours when we have infinitely many contours of the same size. However, in \cite{Rozi4} Rozikov studied an example of an Ising model type on $\mathbb{Z}$ where we have infinitely many contours of size $2$ involving the vertex $0$. He adapted the Peierls argument to prove the phase transition for the model. In this case $w(C)$ must depend on the position of the contour $C$ in the graph, this is the usual situation when the hamiltonian of the model it is not translation invariant, see \cite{BC, Pf}.

\section*{Acknowledgments}

R. Bissacot and E. O. Endo thank Professor Nicolau C. Saldanha, who
pointed out the connection between the results and the Catalan numbers
in an earlier version of this manuscript. They thank A. Procacci and
U. Rozikov for references and comments. They also thank Paulo A. da
Veiga and the organizers of the meetings "\textit{New interactions of
Combinatorics and Probability}" and "\textit{4th Workshop in Stochastic
Modeling}" where they had the opportunity to discuss this note with
colleagues at ICMC-USP and UFSCAR in S\~ao Carlos, Brazil.


\begin{thebibliography}{}

\bibitem{Alon85} N. Alon. An extremal problem for sets with applications
to graph theory. \textit{Journal of Combinatorial Theory, Series
A. }\textbf{40}, 82-89, (1985).

\bibitem{APS12} R. G. Alves, A. Procacci and R. Sanchis. Percolation
on infinite graphs and isoperimetric inequalities.  \emph{Journal of
Statistical Physics.} \textbf{149}, 831--845, (2012).

\bibitem{BB99} E. Babson and I. Benjamini. Cut sets and normed cohomology
with applications to percolation. \emph{Proceedings of the American
Mathematical Society.} \textbf{127}, 589--597, (1999).

\bibitem{BeB99} I. Benjamini and O. Schramm. Percolation beyond
$\mathbb{Z}^d$, many questions and a few answers. \textit{Electronic
Communications in Probability}, \textbf{1}, 71-82, (1996).

\bibitem{BC} R. Bissacot and L. Cioletti. Phase Transition in
Ferromagnetic Ising Models with Non-uniform External Magnetic
Fields. \textit{Journal of Statistical Physics}, \textbf{139}, Issue 5,
769--778, (2010).

\bibitem{Bollo} B. Bollob\'{a}s. On generalized graphs. \textit{Acta
Mathematica Hungarica}. \textbf{16}, 447-452, (1965).

\bibitem{Bo} A. Bovier. {\it Statistical Mechanics of Disordered
Systems,
A Mathematical Perspective}. Cambridge University Press (2012).

\bibitem{BB07} P. N. Balister and B. Bollob{\'a}s. Counting regions with
bounded surface area.  \emph{Communications in Mathematical Physics.}
\textbf{273}, 305--315, (2007).

\bibitem{FS09} P. Flajolet and R. Sedgewick. {\it Analytic Combinatorics}.
Cambridge University Press, Cambridge. (2009).

\bibitem{GJRSS} D. Gandolfo,  J. Ruiz and S. Shlosman. A Manifold
of Pure Gibbs States of the Ising Model on the Lobachevsky
Plane. \emph{Communications in Mathematical Physics.} \textbf{334},
313---330, (2015).

\bibitem{Ge} H.-O. Georgii. {\it Gibbs Measures and Phase Transitions}. de
Gruyter, Berlin, (1988).

\bibitem{JS} J. Jonasson and J. E. Steif. Amenability and Phase
Transition in the Ising Model. \textit{Journal of Theoretical
Probability}, \textbf{12}, 549--559, (1999).

\bibitem{LM98} J. L. Lebowitz and A. E. Mazel. Improved {P}eierls
argument for high-dimensional {I}sing  models.\emph{Journal of Statistical
Physics.} \textbf{90}, 1051--1059, (1998).

\bibitem{Pe} R. Peierls. On Ising's model of ferromagnetism.
\emph{Proceedings of the Cambridge Philosophical Society.} \textbf{32},
477-481, (1936).

\bibitem{Pf} C.-E. Pfister. Large deviations and phase separation in the two-dimensional Ising model. \textit{Helvetica Physica Acta} \textbf{64}, 953-1054, (1991).

\bibitem{FiSe} F. Rassoul-Agha and T. Sepp\"{a}l\"{a}inen. {\it A course
on large deviations with an introduction to Gibbs measures}. Graduate
Studies in Mathematics, \textbf{162}, American Mathematical Society,
Providence (2015).

\bibitem{Rozi} U. A. Rozikov. Gibbs measures on Cayley trees: results and
open problems. \emph{Reviews in Mathematical Physics}, \textbf{25}, No.1,
(2013).

\bibitem{Rozi1} U.A. Rozikov. On q-component models on Cayley tree:
contour method. \emph{Letters in Mathematical Physics}, \textbf{71},
No. 1, 27--38, (2005).

\bibitem{Rozi2} U.A. Rozikov. A Contour Method on Cayley
Trees. \emph{Journal of Statistical Physics}, \textbf{130}, 801-813,
(2008).

\bibitem{Rozi3} U.A. Rozikov. {\it Gibbs Measures on Cayley Trees}.
World Scientific, (2013).

\bibitem{Rozi4} U.A.  Rozikov. An Example of One-Dimensional Phase
Transition. {\it Siberian Advances in Mathematics}, \textbf{16}, No.2,
121-125, (2006)

\bibitem{Ru} D. Ruelle. {\it Statistical Mechanics - Rigorous
Results}. Second Edition. Imperial College Press and World Scientific
Publishing (1999).

\bibitem{S1} R. P. Stanley.  {\it Enumerative Combinatorics}. vol
1. Cambridge University Press, Cambridge (1997).

\bibitem{S} R. P. Stanley. {\it Enumerative Combinatorics}. vol
2. Cambridge University Press, Cambridge (1999).

\end{thebibliography}
\end{document}